\documentclass[11pt,oneside]{article}
\usepackage{color}
\usepackage{xcolor}
\usepackage{psfrag}
\usepackage{epsfig}
\usepackage{graphicx}
\usepackage{amsmath}
\usepackage{amssymb}
\usepackage{amsthm}
\usepackage{enumerate}
\usepackage{cite}
\usepackage[utf8x]{inputenc}

\newtheorem{theorem}{Theorem}[section]

\newtheorem{fact}[theorem]{Fact}

\newtheorem{proposition}[theorem]{Proposition}

\newtheorem{conjecture}[theorem]{Conjecture}
\newtheorem{corollary}[theorem]{Corollary}
\newtheorem{remark}[theorem]{Remark}
\newcommand{\numberset}{\mathbb}

\newcommand{\F}{\numberset{F}}
\renewcommand{\S}{\mathcal{S}}
\renewcommand{\P}{\mathbb{P}}
\newcommand{\X}{\mathcal{X}}

\begin{document}

	\title{Intersections between the norm-trace curve and some low degree curves}
	\date{}
	\author{Matteo Bonini, Massimiliano Sala}
	
	\maketitle
	
	\begin{abstract}
		In this paper we analyze the intersection between the norm-trace curve over $\F_{q^3}$ and the curves of the form $y=ax^3+bx^2+cx+d$, giving a complete characterization of the intersection between the curve and the parabolas (a=0), as well as sharp bounds for the other cases. This information is used for the determination of the weight distribution of some one-point AG codes arising from the curve.
	\end{abstract}
	
	
	
	\section{Introduction}
	Algebraic Geometry (AG codes for short) codes form an important class of error correcting codes; see \cite{Goppa81,Goppa82,S2009}. 
	
	Let $\mathcal{X}$ be an algebraic curve defined over the finite field $\mathbb{F}_q$ of order $q$. The parameters of the AG codes associated with $\mathcal{X}$ strictly depend on some properties of the underlying curve $\mathcal{X}$. In general, curves with many $\mathbb{F}_q$-rational places with respect to their genus give rise to AG codes with good parameters. For this reason, maximal curves, that is, curves attaining the Hasse-Weil upper bound, have been widely investigated in the literature; see for instance \cite{Stichtenoth1988,Tiersma1987,XL2000,YK1992,BMZ2,BMZ1,M2004}.
	
	The determination of the intersection of a given curve $\X$ and low degree curves is often useful for the
	determination of the weight distribution of the AG code arising from $\X$; see \cite{BR2013,BB2018,C2012,MPS2014,MPS2016}.
	
	The norm-trace curve is a natural generalization of the Hermitian curve to any extension field $\F_{q^r}$. It has been widely studied for coding theoretical purposes; see \cite{G2003,BR2013,KL2018,FMT2013,MTT2008,MR2018}.
	
	In this paper, we focus on the intersection between the norm-trace curves and curves of the form $y = ax^3 + bx^2 + cx + d$ over $\F_{q^3}$. We characterize the intersection between the norm-trace curve and parabolas and we provide tools to get sharp bounds in the other cases. To do so we investigate specific irreducible
	surfaces over finite fields. In addition, we partially deduce the weight distribution of the corresponding one-point codes.
	
	\section{Preliminary Results}
	
	Throughout the paper, let $p$ be a  prime and $q=p^m$, where $m$ is a positive integer. Let $\mathbb{F}_q$ be the finite field with $q$ elements. An $[n,k,d]$ linear code $\mathcal{C}$ over $\mathbb{F}_q$ is a $k$-dimensional subspace of $\mathbb{F}_q^n$ with minimum Hamming distance $d$.  Let $A_i$ be the number of codewords with Hamming weight $i$ in $\mathcal{C}$.
	
	\subsection{The norm-trace curve}
	The \emph{norm-trace curve} $\mathcal{X}$ is the plane curve defined over $\F_{q^r}$ by the affine equation 
	\begin{equation}
	\label{eq:NormTrace}
	x^{\frac{q^r-1}{q-1}}=y^{q^{r-1}}+y^{q^{r-2}}+\dots+y^q+y.
	\end{equation}
	The \textit{norm} $\text{N}_{\F_q}^{\F_{q^r}}$ and the \textit{trace} $\text{T}_{\F_q}^{\F_{q^r}}$ are two well-known functions from $\F_{q^r}$ to $\F_q$ such that
	\[
	\text{N}_{\F_q}^{\F_{q^r}}(x)=x^{\frac{q^r-1}{q-1}}=x^{q^{r-1}+q^{r-2}+\dots+q+1}
	\]
	and
	\[
	\text{T}_{\F_q}^{\F_{q^r}}(x)=x^{q^{r-1}}+x^{q^{r-2}}+\dots+x^q+x.
	\]
	When $q$ and $r$ are understood, we will write $\text{N}=\text{N}_{\F_q}^{\F_{q^r}}$ and $\text{T}=\text{T}_{\F_q}^{\F_{q^r}}$.
	
	The equation $x^{\frac{q^r-1}{q-1}}=y^{q^{r-1}}+y^{q^{r-2}}+\dots+y^q+y$ has precisely $q^{2r-1}$ solutions in $\mathbb{A}^2\left(\F_{q^r}\right)$, so the curve $\mathcal{X}$ has $q^{2r-1}+1$ rational points: $q^{2r-1}$ of them correspond to affine places, plus a single place at the infinity $P_{\infty}$.
	
	If $r=2$, $\mathcal{X}$ coincides with  the Hermitian curve and if $r\ge3$ $\mathcal{X}$ is singular in $P_{\infty}$.

	Moreover it is known that its Weierstrass semigroup in $P_{\infty}$ is generated by $\left\langle q^{r-1}, \frac{q^r-1}{q-1}\right\rangle$, see \cite{G2003}.
	
	Our main aim is the study of the intersection between $\mathcal{X}$ and cubics of the form $y=ax^3+bx^2+cx+d$, where $a,b,c,d\in\F_{q^r}$. In particular, we focus on the intersections between $\mathcal{X}$ and parabolas. The case $r=2$ and $a=0$ is completely investigated in \cite{MPS2014,DDK209}, so we deal with the more difficult case $r\ge3$. We set the problem for the general case in Section 3, while in the rest of the paper we concentrate on the solution to the case $r=3$ and $a=0$, obtaining partial results for the case $a\ne0$ in Section 6.
	
	\subsection{Algebraic Geometry Codes}
	
	We introduce here some basic notions on AG codes; for a detailed introduction to this topic, we refer to \cite[Chapter 2]{S2009}.
	
	Let $\mathbb{F}_q$ be the finite field with $q$ elements and $\mathcal{X}$ be a projective, absolutely irreducible, algebraic curve of genus $g$ defined over $\mathbb{F}_q$.
	Let $\mathbb{F}_q(\mathcal{X})$ be the field of rational functions on $\mathcal{X}$ and $\mathcal{X}(\mathbb{F}_q)$ be the set of rational places of $\mathcal{X}$.
	For any divisor $D=\sum_{P\in\mathcal{X}(\overline{\mathbb{F}}_q)}m_P P$ on $\mathcal{X}$, we denote by $v_P(D)$ the valuation $m_P\in\mathbb{Z}$ of $D$ at $P$, and by ${\rm supp}(D)$ the support of $D$; the degree of $D$ is $\deg(D)=\sum_{P\in{\rm supp}(D)} n_P$.
	The Riemann-Roch space $\mathcal{L}(D)$ of an $\mathbb{F}_q$-rational divisor $D$ is the $\mathbb{F}_q$-vector space 
	\[
	\mathcal{L}(D)=\{f\in\F_{q}(\mathcal{X}) \ |\ (f)+D\ge0\}\cup \{0\}.
	\]
	where $(f)=(f)_0-(f)_{\infty}$ denotes the principal divisor of $f$; here, $(f)_0$ and $(f)_\infty$ are respectively the zero divisor and the pole divisor of $f$.
	The $\mathbb{F}_q$-dimension of $\mathcal{L}(D)$ is denoted by $\ell(D)$. It is known that $\mathcal{L}(D)$ is a finite dimensional $\F_q$-vector space and the exact dimension can be computed using the Riemann-Roch theorem. The $\mathbb{F}_q$-dimension of $\mathcal{L}(D)$ is denoted by $\ell(D)$.
	
	Consider now the divisor $D=\sum_{P\in S}P$, $S=\{P_1,\dots,P_n\}\subsetneq\mathcal{X}(\F_{q})$, where all the $P$'s have valuation one. Let $G$ be another $\F_q$-rational divisor such that $\mathrm{supp}(G)\cap \mathrm{supp}(D)=\emptyset$. Consider the evaluation map 
	
	\[
	\mathrm{ev}:\mathcal{L}(G)\rightarrow (\F_q)^n\quad,\quad \mathrm{ev}(f)=(f(P_1),\dots,f(P_n)).
	\]
	This map is $\F_q$-linear and it is injective if $n>\deg(G)$.
	
	The AG-code $C_{\mathcal{L}}(D,G)$ associated with the divisors $D$ and $G$ is then defined as $\mathrm{ev}(\mathcal{L}(G))$. It is well known that $\ell(G)>\ell(G-D)$ and that $C_{\mathcal{L}}(D,G)$ is an $[n,\ell(G)-\ell(G-D),d]_q$ code, where $d\ge d^*=n-\deg(G)$, and $d^*$ is the so called \textit{designed minimum distance} of the code. 
	
	\section{Intersections between $\mathcal{X}$ and a curve $y=A(x)$ of degree $h$}
	Our aim is to find out the intersection over $\F_{q^3}$ of $\mathcal{X}$ with the curve defined by the polynomial $y=A(x)$ of degree $h$, so $A(x)=A_{h}x^h+\dots+A_0$, where $A_h\ne0$ and $A_i\in \F_{q^r}$.
	More precisely, given two curves $\mathcal{X}$ and $\mathcal{Y}$ lying in the affine plane $\mathbb{A}^2(\F_{q^r})$ we call \textit{planar intersection} (or simply intersection) the number of points in $\mathbb{A}^2(\F_{q^r})$ that lie in both curves, disregarding multiplicity.
	Substituting $y=A(x)$ in the equation of the norm-trace curve, we get, by the linearity of T, 
	\[
	\text{N}(x)=\text{T}(A_hx^h)+\dots+\text{T}(A_1x)+\text{T}(A_0).
	\]
	Given a basis $\mathcal{B}=\{w_0,\dots,w_{r-1}\}$ of $\F_{q^r}$ over $\F_q$, we know that there is a vector space isomorphism $\Phi_{\mathcal{B}}:(\F_q)^r\rightarrow\F_{q^r}$ such that $\Phi_{\mathcal{B}}(s_0,\dots, s_{r-1})=\sum_{i=0}^{r-1}s_iw_i$. 
	
	The maps $\mathrm{N,T}:\F_{q^r}\rightarrow\F_q$ can be seen as maps from $(\F_q)^r$ to $\F_q$, identifying $\widetilde{\text{N}}=\text{N}\circ\Phi_{\mathcal{B}}$ and $\widetilde{\text{T}}=\text{T}\circ\Phi_{\mathcal{B}}$ with $\mathrm{N}$ and $\mathrm{T}$. Also, we can consider $\text{T}_i:=\text{\text{T}}(A_ix^i)$  and $\widetilde{\text{T}}_i:=\text{T}_i\circ\Phi_{\mathcal{B}}$, $1\le i\le h$.
	From now on, we will take as $\mathcal{B}$ a normal basis, i.e. a basis $\mathcal{B}=\{\alpha,\alpha^q,\dots,\alpha^{q^{r-1}}\}$, for some $\alpha\in\F_{q^3}$. We know that such a basis exists, see \cite[Theorem 2.35]{LN1997}. A simple manipulation shows that $\widetilde{\text{N}}$ and $\widetilde{\text{T}}_i$ are homogeneous polynomials of degree respectively $r$ and $i$ in $\F_q[x_0,\dots,x_{r-1}]$. Therefore
	\begin{equation}
	\label{eq:superficie}
	\widetilde{\text{N}}(x_0,\dots,x_{r-1})=\widetilde{\text{T}}_h(x_0,\dots,x_{r-1})+\dots+\widetilde{\text{T}}_1(x_0,\dots,x_{r-1})+D
	\end{equation}
	which is the equation of a hypersurface of $\mathbb{A}^r(\overline{\F}_q)$, where $D=\text{T}(A_0)$.
	Notice that the LHS has degree $r$, while the RHS has degree $h$.
	
	\section{Case $r=3$ and $h=2$}
	We are interested in this case to find the number of possible intersections between the norm-trace curve and the parabolas. By parabola we mean a curve $y=Ax^2+Bx+C$, $A,B,C\in\F_{q^3}$ and $A\ne0$. These numbers help to determine some weights for the corresponding AG code, see Section~6. From now on, $\mathcal{B}=\{\alpha,\alpha^q,\alpha^{q^2}\}$.
	
	Specializing to $y=Ax^2+Bx+C$, Equation (\ref{eq:superficie}) reads
	\begin{equation}
	\label{eq:superficie3}
	\widetilde{\text{N}}(x_0,x_1,x_{2})=\widetilde{\text{T}}_2(x_0,x_1,x_{2})+\widetilde{\text{T}}_1(x_0,x_1,x_{2})+D.
	\end{equation}
	
	The map $\Phi^{-1}_{\mathcal{B}}:\F_{q^3}\rightarrow(\F_q)^3$ induces a correspondence between $\F_q[x_0,x_1,x_2]$ and $\F_{q^3}[x]$ such that we can substitute $x$ with $x_0\alpha+x_1\alpha^q+x_2\alpha^{q^2}$ and $x^2$ with
	\[ 
	x_0^2\alpha^2+x_1^2\alpha^{2q}+x_2^2\alpha^{2q^2}+2x_0x_1\alpha^{q+1}+2x_0x_2\alpha^{q^2+1}+2x_1x_2\alpha^{q^2+q}.
	\]
	Using this relation we want to write down the explicit equation of the surface (\ref{eq:superficie3}) of $\mathrm{AG(3,q)}$.
	
	\[
	\begin{split}
	\widetilde{\text{T}}_1&=B(x_0\alpha+x_1\alpha^q+x_2\alpha^{q^2})+B^q(x_0\alpha^q+x_1\alpha^{q^2}+x_2\alpha)+B^{q^2}(x_0\alpha^{q^2}+x_1\alpha+x_2\alpha^{q})\\
	&=x_0\text{T}(\alpha B)+x_1\text{T}(\alpha B^{q^2})+x_2\text{T}(\alpha B^q),
	\end{split}
	\]
	
	\[
	\begin{split}
	\widetilde{\text{T}}_2=&A(x_0\alpha+x_1\alpha^{q}+x_2\alpha^{q^2})^2+A^q(x_0\alpha^q+x_1\alpha^{q^2}+x_2\alpha)^{2}+A^{q^2}(x_0\alpha^{q^2}+x_1\alpha+x_2\alpha^{q})^{2}\\
	=&x_0^2\text{T}(A\alpha^2)+x_1^2\text{T}(A\alpha^{2q})+x_2^2\text{T}(A\alpha^{2q^2})+2x_0x_1\text{T}(A\alpha^{q+1})+2x_0x_2\text{T}(A\alpha^{q^2+1})+2x_1x_2\text{T}(A\alpha^{q^2+q}),\\
	\end{split}
	\]
	
	\[
	\begin{split}
	\widetilde{\text{N}}=&(x_0\alpha^{q^2}+x_1\alpha+x_2\alpha^{q})(x_0\alpha^q+x_1\alpha^{q^2}+x_2\alpha)(x_0\alpha+x_1\alpha^{q}+x_2\alpha^{q^2})\\
	=&(x_0^3+x_1^3+x_2^3)\text{N}(\alpha)+(x_0^2x_1+x_1^2x_2+x_2^2x_0)\text{T}(\alpha^{q+2})+(x_0^2x_2+x_1^2x_0+x_2^2x_1)\text{T}(\alpha^{2q+1})\\&+x_0x_1x_2(3\text{N}(\alpha)+\text{T}(\alpha^3)).\\
	\end{split}
	\]
	
	Therefore (\ref{eq:superficie3}) reads
	
	\begin{equation}
	\label{EqSuperficie1}
	\begin{split}
	0=&-(x_0^3+x_1^3+x_2^3)\text{N}(\alpha)-(x_0^2x_1+x_1^2x_2+x_2^2x_0)\text{T}(\alpha^{q+2})-(x_0^2x_2+x_1^2x_0+x_2^2x_1)\text{T}(\alpha^{2q+1})\\&-x_0x_1x_2(3\text{N}(\alpha)+\text{T}(\alpha^3))+x_0^2\text{T}(A\alpha^2)+x_1^2\text{T}(A\alpha^{2q})+x_2^2\text{T}(A\alpha^{2q^2})+2x_0x_1\text{T}(A\alpha^{q+1})\\&+2x_0x_2\text{T}(A\alpha^{q^2+1})+2x_1x_2\text{T}(A\alpha^{q^2+q})+x_0\text{T}(\alpha B)+x_1\text{T}(\alpha B^{q^2})+x_2\text{T}(\alpha B^q)+D.
	\end{split}
	\end{equation}
	Denote by $\S_1$ the surface defined by the polynomial above. Note that $\S_1$ is defined over $\F_q$. For a given surface, let $\mathcal{X}(\F_q)$ be the set of its $\F_q$-rational points.
	\begin{remark}
		By construction, the $\F_q$-rational points of $\S_1$, i.e. the points in $\S_1(\F_q)$, correspond to the intersections in $\mathbb{A}^2(\F_{q^3})$ between the norm-trace curve and the parabola $y=Ax^2+Bx+C$. This happens because (\ref{EqSuperficie1}) comes from a sequence of manipulations that started with $\mathrm{N}(x)=\mathrm{Tr}(Ax^2+Bx+C)$, i.e. there exists an $x\in\F_{q^3}$ such that $\mathrm{N}(x)=\mathrm{Tr}(Ax^2+Bx+C)$ if and only if there exist $(x_0,x_1,x_2)\in(\F_q)^3$ that satisfies (\ref{EqSuperficie1}) and $x=x_0\alpha+x_1\alpha^q+x_2\alpha^{q^2}$.
	\end{remark}
	
	Equation $(\ref{EqSuperficie1})$ can be also written as 
	\begin{equation*}
	\begin{split}
	0=&-(x_0\alpha+x_1\alpha^q+x_2\alpha^{q^2})(x_0\alpha^q+x_1\alpha^{q^2}+x_2\alpha)(x_0\alpha^{q^2}+x_1\alpha+x_2\alpha^q)+A(x_0\alpha+x_1\alpha^q+x_2\alpha^{q^2})^2\\&+A^q(x_0\alpha^q+x_1\alpha^{q^2}+x_2\alpha)^{2}+A^{q^2}(x_0\alpha^{q^2}+x_1\alpha+x_2\alpha^q)^{2}+B(x_0\alpha+x_1\alpha^q+x_2\alpha^{q^2})\\&+B^q(x_0\alpha^q+x_1\alpha^{q^2}+x_2\alpha)+B^{q^2}(x_0\alpha^{q^2}+x_1\alpha+x_2\alpha^q)+D.
	\end{split}
	\end{equation*}

	Consider the non-singular matrix (since we are dealing with three linearly independent elements, see \cite[Corollary 2.38]{LN1997})
	\[
	M=\begin{pmatrix}
	\alpha&\alpha^q&\alpha^{q^2}\\
	\alpha^q&\alpha^{q^2}&\alpha\\
	\alpha^{q^2}&\alpha&\alpha^q
	\end{pmatrix}
	\]  
	and the affine change of variables in $\mathbb{A}^3$ defined by $\psi(x_0,x_1,x_2)=M(x_0,x_1,x_2)^t.$
	Let $\S_2$ be the corresponding surface obtained from $\S_1$. Then  $\S_2$ is defined over $\F_{q^3}$, and has equation
	\begin{equation}
	\label{EqSuperficie2}
	X_0X_1X_2=AX_0^2+A^qX_1^2+A^{q^2}X_2^2+BX_0+B^qX_1+B^{q^2}X_2+D.
	\end{equation}

	\begin{remark}
		\label{remark:senso}
		Clearly, $\F_{q}$-rational points of $\mathcal{S}_1$ are mapped to $\F_{q^3}$-rational points of $\mathcal{S}_2$ of the form $(\beta,\beta^q,\beta^{q^2})$, $\beta\in\F_{q^3}$, and viceversa. Moreover, $\psi$ preserves number and degree of any absolutely
		irreducible component of $\S_1$ and its singularities.
	\end{remark}
	
	\begin{proposition}
		$\mathcal{S}_1$ is an  absolutely irreducible cubic surface.
	\end{proposition}
	\begin{proof}
		By Remark \ref{remark:senso} it is sufficient to prove that $\mathcal{S}_2$ is absolutely irreducible. We proceed by contradiction. If $\mathcal{S}_2$ is reducible, since its degree is three, then it must contain a plane. In this case we would have
		
		{\footnotesize
			\begin{equation}
			\label{prodotto:irrid}
			q(X_0,X_1,X_2)(k_0X_0+k_1X_1+k_2X_2+k_3)=X_0X_1X_2-AX_0^2-A^qX_1^2-A^{q^2}X_2^2-BX_0-B^qX_1-B^{q^2}X_2-D
			\end{equation}
		}
		where $q(x_0,x_1,x_2)$ is the equation of a quadric surface, $k_j\in\overline{\F}_q$, $j\in\{0,\dots,3\}$, and at least one of $k_0,k_1,k_2$ is nonzero.
		
		Consider the intersections of $\S_2$ with the plane at the infinity, then
		\[
		Q(X_0,X_1,X_2)(k_0X_0+k_1X_1+k_2X_2)=X_0X_1X_2
		\]
		where $Q(X_0,X_1,X_2)$ is the polynomial given by the degree 2 terms of $q(X_0,X_1,X_2)$. This expression implies that two among $k_0,k_1,k_2$ have to be zero, and then the plane has equation $k_iX_i+k_3=0$ for $i\in\{0,1,2\}$. Suppose, whitout loss of generality that $i=0$, then Equation (\ref{prodotto:irrid}) reads \[
		\begin{split}
		q(X_0,X_1,X_2)(X_0+k)=X_0X_1X_2-AX_0^2-A^qX_1^2-A^{q^2}X_2^2-BX_0-B^qX_1-B^{q^2}X_2-D
		\end{split}
		\]
		for a given $k\in\overline{\F}_q$.
		
		Applying the identity principle for polynomials, we obtain that $q(X_0,X_1,X_2)=X_1X_2+h_0X_0+h_1X_1+h_2X_2+h_3$, where $h_i\in\overline{\F}_q$, $i\in\{0,\dots,3\}$. At this point Equation (\ref{prodotto:irrid}) becomes
		{\footnotesize\[
			(X_1X_2+h_0X_0+h_1X_1+h_2X_2+h_3)(k_0X_0+k_3)=X_0X_1X_2-AX_0^2-A^qX_1^2-A^{q^2}X_2^2-BX_0-B^qX_1-B^{q^2}X_2-D
			\]}
		and since $A^q\ne0$ this cannot happen.

	\end{proof}
	
	What we want to do now is to estimate the number of $\F_q$-rational points of $\mathcal{S}_1$. Since they correspond to the intersections between $\mathcal{X}$ and $y=Ax^2+Bx+C$, by applying the B\'ezout theorem we get that
	\[
	|\mathcal{S}_1(\F_q)|\le 2(q^2+q+1).
	\]
	This bound can be improved, as we will see. A better estimate can be obtained using the Lang-Weil bound.
	\begin{theorem}[\cite{LW1954}]
		Given nonnegative integers $n, d$ and $r$, with $d > 0$, there is a
		positive constant $A(n,d,r)$ such that for every finite field $\F_q$, and every	irreducible subvariety $\mathcal{X}\subseteq\mathbb{P}^n(\F_q)$ of dimension $r$ and degree $d$, we have
		\[
		||\mathcal{X}(\F_q)|-q^r|\le(d-1)(d-2)q^{r-\frac{1}{2}}+A(n,d,r)q^{r-1}.
		\]
	\end{theorem}
	
	\begin{corollary}
		The number of $\F_q$-rational points on the surface $\mathcal{S}_1(\F_{q})$ is bounded by 
		\[q^2+2q^\frac{3}{2}+A(3,3,2)q.\]
	\end{corollary}
	
	This bound improves Bézout’s Theorem. Also, other theoretical estimates are known; see \cite{B2013}. 
	In what follows we will provide a bound of the type
	\[\S_1(\F_{q})\le q^2+\eta q+\mu\]
	where $\mu<q$ and $\eta$ is upper bounded by a constant (independent from $q$ and $\mu$). Experimentally we found the following
	\begin{fact}
		For $q\in\{2,\dots,29\}$, $|\eta|\le2$ and $\mu=1$.
	\end{fact}
	\begin{conjecture}
		$|\eta|\le2$ and $\mu=1$ for all $q$.
	\end{conjecture}
	
	Recall some previous results
	
	\begin{theorem}[\cite{M1986}, Theorem 27.1]
		Let $\mathcal{S}$ be a cubic surface over $\F_q$. If $\mathcal{S}$ is birationally trivial (i.e. to allow a $\F_q$-birational map to $\mathbb{P}^2(\F_q)$), then \[|S(\F_q)|\equiv1 \mod q.\]
	\end{theorem}
	
	In the case in which $S_1$ is smooth we also know the possible values for $|S_1(\F_q)|$.
	
	\begin{theorem}[\cite{M1986}, Theorem 23.1]
		\label{Weil}
		Let $\mathcal{S}$ be a smooth irreducible cubic surface over $\F_q$, then the number of points of $\mathcal{S}(\F_q)$ is exactly
		\[
		|\mathcal{S}(\F_q)|= q^2+\eta q+1
		\]
		where $\eta\in\{-2,-1,0,1,2,3,4,5,7\}$.
	\end{theorem}
	
	In view of Theorem \ref{Weil}, we consider separately
	the cases $\S_1$ smooth and $\S_1$ singular.
	
	\section{Singular case}
	From now on we investigate when $\mathcal{S}_1$ is singular. We start by observing that the possible singular points can only be double points, since $\S_1$ is a cubic irreducible surface.  Moreover, recalling that an isolated singularity $P_s$ means that there exists a neighbourhood containing only $P_s$ as singular point, we will see that $\S_2$ has only isolated singularities. This happens beacuse the ideal defined by its equation and the partial derivatives is zero-dimensional (as it is possible to see in the proof of Proposition \ref{caso:singolari4}). In this context the following result is very helpful.
	\begin{theorem}[\cite{CT1988}]
		Let $\mathcal{S}\subset\mathbb{P}^3(\mathbb{K})$ be a singular irreducible cubic surface defined on the field $\mathbb{K}$. Let $\bar{\S}=\S(\overline{\mathbb{K}})$ be the surface defined by $\S$ over $\overline{\mathbb{K}}$, the algebraic closure of $\mathbb{K}$. Let $\delta$ be the number of isolated double points of $\bar{\S}$. Then $\delta\le 4$ and $\mathcal{S}$ is birationally equivalent (over $\mathbb{K}$) to
		\begin{enumerate}[(i)]
			\item $\P^2(\mathbb{K})$ if $\delta=1,4$;
			\item a smooth Del Pezzo surface of degree 4 if $\delta=2$;
			\item a smooth Del Pezzo surface of degree 6 if $\delta=3$.
		\end{enumerate}
	\end{theorem}
	Recall that a smooth Del Pezzo surface is a smooth projective surface $V$ whose
	anticanonical class is ample. Many arithmetic properties of these surfaces were
	investigated by Manin; see \cite{M1986}.

	What we want to do now is to find a bound of type $q^2+\eta q+\mu$ for the four possible cases of singularities ($\delta=1,2,3,4$).
	
	Clearly the affine singular points on $\mathcal{S}_2$ correspond to the solutions of 
	\begin{equation}
	\label{derivate:sup}
	\begin{cases}
	X_0X_1X_2=AX_0^2+A^qX_1^2+A^{q^2}X_2^2+BX_0+B^qX_1+B^{q^2}X_2+D\\
	X_1X_2=2AX_0+B\\
	X_0X_2=2A^qX_1+B^q\\
	X_0X_1=2A^{q^2} X_2+B^{q^2}
	\end{cases}	
	\end{equation}

	\begin{remark}
		$\S_2$ has no singular point at the infinity.
	\end{remark}
	\begin{proof}
		A straightforward computation shows that the singular points at the infinity of $\S_2$ satisfy the following system of equations
		\[
		\begin{cases}
		X_0X_1X_2=0\\
		X_1X_2=0\\
		X_0X_2=0\\
		X_0X_1=0\\
		AX_0^2+A^qX_1^2+A^{q^2}X_2^2=0
		\end{cases}	
		\]
		which admits only $(0:0:0:0)$ as solution, which is not a point of the projective space.
	\end{proof}

	\begin{remark}
		\label{remark:singolarirazionali}
		Since $\mathcal{S}_1$ is defined over $\F_q$ if $P\in{\mathcal{S}}_1(\F_q)$ is a singular point then its conjugates with respect to the Frobenius automorphism are also singular. 
	\end{remark}
	
	\begin{remark}
		Notice also that if a singular point of $\S_2$ is $\F_{q^6}$-rational the corresponding singularity of $\S_1$ will be $\F_{q^2}$-rational since $(x_0\alpha+x_1\alpha^q+x_2\alpha^{q^2})^{q^6}=(x_0\alpha+x_1\alpha^q+x_2\alpha^{q^2})^{q^2}$.
	\end{remark}
	
	Before considering the classification of the four cases arising from different values of $\delta$, we need to examine separately the case $B=0$, which turns out ot to be special.
	
	\subsection{Case B=0}
	In this case the singularities of the surface correspond to the solutions of 
	\begin{equation}
	\label{sistema:b0}
	\begin{cases}
	X_0X_1X_2=AX_0^2+A^qX_1^2+A^{q^2}X_2^2+D\\	
	X_1X_2=2AX_0\\
	X_0X_2=2A^qX_1\\
	X_0X_1=2A^{q^2} X_1\\
	\end{cases}
	\end{equation}
	
	A direct computation leads to the fact that if $(\bar{x}_0,\bar{x}_1,\bar{x}_2)\ne(0,0,0)$ is a singular point, then each $\bar{x}_i$ is different from zero.
	
	\begin{proposition}
		The only possible singularities for the case $B=0$ are described in the following list:
		\begin{enumerate}[(i)]
			\item if $D=0$ then, for any $q$, $\S_2$ admits only $P=(0,0,0)$ as singular point;
			\item if $D\ne0$ and $q$ is odd then $\delta=4$ and the singular points are given by $(\gamma,\gamma^q,\gamma^{q^2})$, $(\gamma,-\gamma^q,-\gamma^{q^2})$, $(-\gamma,\gamma^q,-\gamma^{q^2})$, $(-\gamma,-\gamma^q,\gamma^{q^2})$, where $\gamma=2A^{\frac{q^2+q}{2}}$.
		\end{enumerate}
	\end{proposition}
	
	\begin{proof}
		Direct computations show that $(i)$ comes from Equation (\ref{sistema:b0}), so we are left with the case $q$ odd and $(0,0,0)$ not singular. Substituting the derivatives into the equation that defines the surface we get
		\[
		\begin{cases}
		-AX_0^2+A^qX_1^2+A^{q^2}X_2^2+D=0\\
		AX_0^2-A^qX_1^2+A^{q^2}X_2^2+D=0\\
		AX_0^2+A^qX_1^2-A^{q^2}X_2^2+D=0.
		\end{cases}
		\]
		
		Summing pairwise the equations gives us
		
		\[
		\begin{cases}
		2AX_0^2+2D=0\\
		2A^qX_1^2+2D=0\\
		2A^{q^2}X_2^2+2D=0
		\end{cases}
		\]
		
		and, since $q$ is odd and $A\ne0$
		
		\begin{equation}
		\label{sist:b=0}
		\begin{cases}
		X_0^2=-\frac{D}{A}\\
		X_1^2=-\frac{D}{A^q}\\
		X_2^2=-\frac{D}{A^{q^2}}
		\end{cases}
		\end{equation}
		
		The fact that $\beta\in\F_{q}$ is a square if and only if $\beta^q$ is a square implies that all the equations of (\ref{sist:b=0}) are solvable if and only if the first one is. Therefore (\ref{sist:b=0}) is solvable if and only if $-\frac{D}{A}$ is a square of $\F_{q^3}$, but this is always true since $-\frac{D}{A}=A^{q^2+q}$ has is an even power of $A$.
		From the equation of the surface it follows that $\S_2$ has four singularities of the form $(\gamma,\gamma^q,\gamma^{q^2})$, $(\gamma,-\gamma^q,-\gamma^{q^2})$, $(-\gamma,\gamma^q,-\gamma^{q^2})$, $(-\gamma,-\gamma^q,\gamma^{q^2})$, where $\gamma=2A^{\frac{q^2+q}{2}}$.
		
	\end{proof}
	
	\begin{remark}
		Notice that in case $D\ne0$ and $q$ odd, the four singular points cannot be all distinct conjugates with respect to the Frobenius automorphism. This comes from the explicit representation that was given above, and from the fact that if $\gamma^q=\pm \gamma$ then each point can have at most one different conjugate.
	\end{remark}
	
	\subsection{One singular point}
	From now on we can consider $B\ne0$. From Remark \ref{remark:singolarirazionali} if $\mathcal{S}_1$ has one singular (double) point $P$ then $P$ has to be $\F_{q}$-rational, otherwise also its conjugate should be singular. Consider now the sheaf of $\F_{q}$-rational lines passing through $P$: each line, not contained in $\S_1(\F_q)$, can intersect $\mathcal{S}_1(\F_q)$ in at most one more point, since $P$ is a double point and $\mathcal{S}_1$ has degree three. So the number of $\F_{q}$-rational points of $\mathcal{S}_1$ is given by
	
	\begin{equation*}
	|S_1(\F_{q})|\le q^2+h(q-1)=q^2+hq-h
	\end{equation*}
	
	where $h$ is the number of lines contained in $\S_1$ and passing through $P$.
	
	\begin{proposition}
		With the same notation as before we have $h=0$.
	\end{proposition}
	\begin{proof}
		We want to give a bound for the maximal number of ($\F_{q}$-rational) lines contained in $\S_1$ and passing through $P$. For simplicity we proceed with the computations on $\S_2$, since the number of these lines will be the same. Suppose that the corresponding singular point $Q$ on $\S_2$ has coordinates $(a,a^q,a^{q^2})$. Then, since it is the only singular point, we have that $Q$ is the only point that satisfies (\ref{derivate:sup}).
		Consider now the sheaf of lines passing through $Q$, which has parametric equation, for $b\ne0$
		\[
		\begin{cases}
		X_0=bt+a\\
		X_1=b^qt+a^q\\
		X_2=b^{q^2}t+a^{q^2}
		\end{cases}
		\]
		and after doing the substitution we get that, if the line is contained into $\S_2$,
		\[
		p_3t^3+p_2t^2+p_1t
		\]
		has to be the zero polynomial in $\F_{q}[t]$, where
		\[
		\begin{split}
		&p_3=b^{q^2+q+1},\\
		&p_2=-Ab^2 - (Ab^2)^q - (Ab^2)^{q^2} + b^{q+1}a^{q^2} + b^{q^2+1}a^q + b^{q^2+q}a,\\
		&p_1=-2Aab - 2(Aab)^q - 2(Aab)^{q^2} - Bb - (Bb)^q - (Bb)^{q^2} + ba^{q^2+q} + b^qa^{q^2+1} + b^{q^2}a^{q+1}.
		\end{split}
		\]
		From the fact that $p_3=0$ we have that $\mathrm{N}(b)$ has to be equal to zero, but this means that $b$ is equal to zero, which is a contradiction.
	\end{proof}
	
	Putting together the previous observations we have the following result.
	
	\begin{proposition}
		If $\S_1$ has one singular $\F_{q}$-rational point then 
		\begin{equation}
		\label{stima1punto}
		|S_1(\F_q)|\le q^2.
		\end{equation}
	\end{proposition}

	\subsection{Two singular points}
	Call $P_1$ and $P_2$ the two singular points of $\mathcal{S}_1$, from Remark \ref{remark:singolarirazionali} there are two possibilities:
	\begin{enumerate}[(i)]
		\item $P_1$ and $P_2$ are $\F_q$-rational;
		\item $P_1$ and $P_2$ are $\F_{q^2}$-rational and conjugates.
	\end{enumerate}
	If $(i)$ happens then we can give similar argumentation as in Section 5.2 and get the bound $|S_1(\F_{q})|\le q^2+q-1$.
	
	We look for a bound when $(ii)$ happens: call $r$ the line passing through $P_1$ and $P_2$, since it fixes the conjugate points then it has to be $\F_q$-rational and moreover this line has to be contained in $\mathcal{S}_1(\F_q)$ since the intersection multiplicity of this line is at least $2$ in both $P_1$ and $P_2$ and the surface has degree 3. Now consider the pencil of planes passing through $r$ and consider the cubic curve $\mathcal{C}$ defined as intersection between any of these planes and $\mathcal{S}_1$. Clearly $\mathcal{C}$ is reducible and there are two possible situations:
	\begin{enumerate}
		\item $\mathcal{C}$ is completely reducible. In this case $\mathcal{C}$ is the product of three lines contained in the surface. Call $s$ and $s^\prime$ the two lines different from $r$: $s$ and $s^\prime$ cannot be $\F_q$-rational since they do not fix the conjugates, so they are $\F_{q^2}$-rational. From the fact that they are contained in $\mathcal{S}_1$ and they pass through conjugate points we have that $s^\prime=s^q$. From this fact we have that the number of $\F_{q}$-rational points on $\mathcal{C}\setminus r$ is 1 and that point is $s\cap s^\prime$.
		
		\begin{center}
			\includegraphics[width=0.3\textwidth]{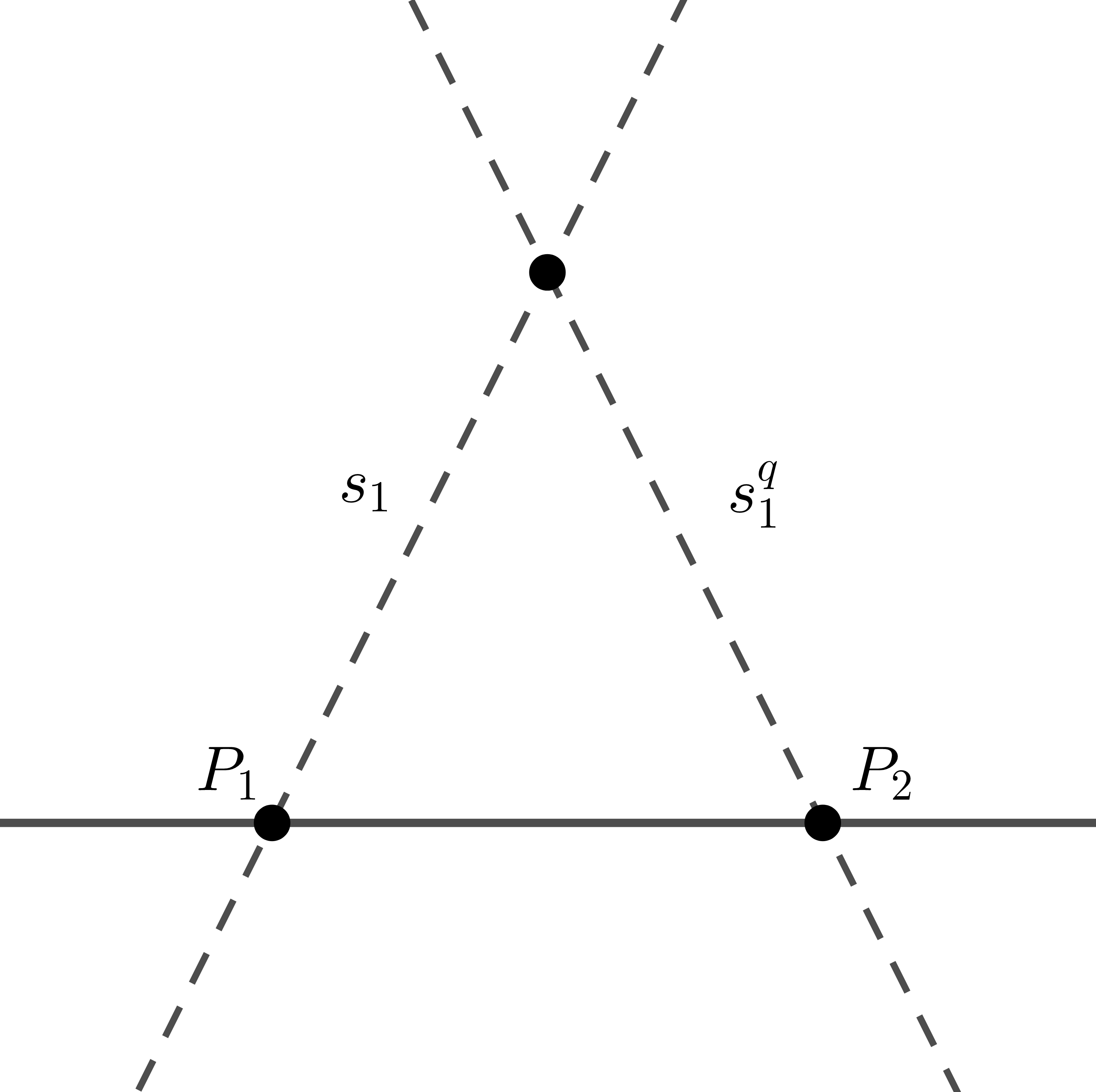}
		\end{center}
		\item $\mathcal{C}$ is the product of $r$ and an irreducible conic $\mathcal{D}$ contained in the surface and it contains exactly $q$ points, see \cite[Lemma 7.2.3]{Hirschfeld1998}. In this case the number of $\F_q$ rational points of $\mathcal{D}$ not contained in $r$ is exactly $q-2$.
		\begin{center}
			\includegraphics[width=0.3\textwidth]{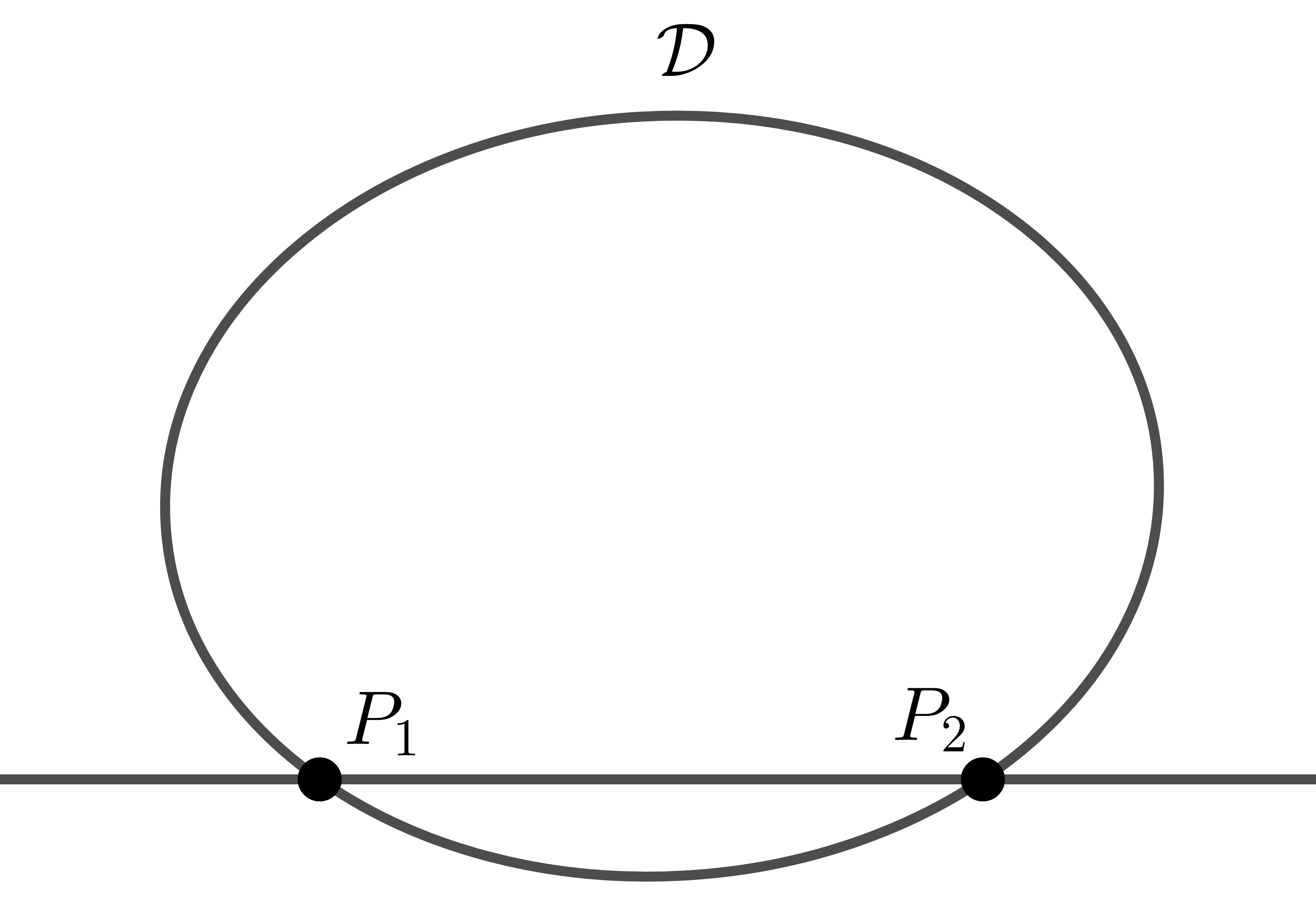}
		\end{center}
	\end{enumerate}
	
	From the analysis of the two possible cases, recalling that the maximum number of lines contained in a cubic surface is 27 (see \cite[Chapter IV]{M1986}), the first situation can happen at most in 13 cases, and so we have:
	
	\[
	q+(q-13)(q-2)+13\le |S_1(\F_q)|\le q(q-2)+q
	\]
	
	Putting together the previous observations we have the following result.
	
	\begin{proposition}
		If $\S_1$ has two singular $\F_{q^2}$-rational conjugate points then 
		\begin{equation}
		\label{stima2punti}
		q^2-14q+39\le |S_1(\F_q)|\le q^2-q.
		\end{equation}
	\end{proposition}

	\subsection{Three singular points}
	Call $P_1$, $P_2$ and $P_3$ the singular points of $\mathcal{S}_1$, from Remark \ref{remark:singolarirazionali} we have the following configurations:
	\begin{enumerate}[(i)]
		\item At least one among $P_1$, $P_2$ and $P_3$ is $\F_{q}$-rational;
		\item $P_1$, $P_2$ and $P_3$ are $\F_{q^3}$-rational and conjugates.
	\end{enumerate}
	If $(i)$ happens then we can give similar argumentation as in Section 5.2 and get the bound $|S_1(\F_{q})|\le q^2+2q-2$.
	
	We start with observing that the three points cannot be collinear, which comes directly from the following proposition.
	\begin{proposition}
		Let $\mathcal{C}$ be a cubic curve such that it has three double points. Then $\mathcal{C}$ is completely reducible and splits in the product of three lines, each passing through a pair of its singular points.
	\end{proposition}
	\begin{proof}
		Direct consequence of B\'ezout's theorem.
	\end{proof}
	In order to get an estimation of $|\S_1(\F_q)|$ for (ii) we change the model of the surface as the following proposition suggests.
	
	\begin{proposition}
		Let $\mathcal{S}$ be a cubic surface over $\P^3(\F_q)$, considered with projective coordinates $[r_0:r_1:r_2:T]$, and such that it has exactly three conjugates $\F_{q^3}$-rational double points, namely $P_1,P_2$ and $P_3$. Then $\mathcal{S}$ is projectively equivalent to the surface having affine equation, for certain $\beta,\gamma\in\F_{q^3}$
		\[
		r_0r_1r_2+\beta r_0r_1+\beta^qr_1r_2+\beta^{q^2}r_0r_2+\gamma r_0+\gamma^qr_1+\gamma^{q^2}r_2=0.
		\]
	\end{proposition}
	\begin{proof}
		Up to a change of projective frame we can consider the following situation
		\begin{itemize}
			\item The plane passing trough the three points is the plane at the infinity $T=0$ and the triangle of lines through them in that plane is given by $r_0$, $r_1$ and $r_2$;
			\item $\mathcal{O}=(0:0:0:1)\in\mathcal{S}$.
		\end{itemize}
		From these choices we obtain the following equation for the surface $\S$
		\[
		r_0r_1r_2+T(\alpha_0r_0^2+\alpha_1r_1^2+\alpha_2r_2^2+\beta_0r_0r_1+\beta_1r_1r_2+\beta_2r_0r_2)+T^2(\gamma_0r_0+\gamma_1r_1+\gamma_2r_2)=0
		\]
		where $\alpha_i,\beta_i,\gamma_i\in\F_{q^3}$ for $i\in\{0,1,2\}$.
		From the fact that $P_1,P_2$ and $P_3$ are conjugates it follows that $r_0,r_1$ and $r_2$ are conjugates and then we get that $\alpha_1=\alpha_0^q$, $\alpha_2=\alpha_0^{q^2}$, $\beta_1=\beta_0^q$, $\beta_2=\beta_0^{q^2}$, $\gamma_1=\gamma_0^q$ and $\gamma_2=\gamma_0^{q^2}$. Consider now the plane $\pi$ passing trough $P_1,P_2$ and $\mathcal{O}$. Without loss of generality, $P_1$ is the singular point satisfying $T=r_1=r_2=0$, then its coordinates will be $P_1=(p_1,p_2,p_3,0)$. Consider now the line, namely $s$ passing through $P_1$ and $\mathcal{O}$. A general point on that line has coordinates $P_{\lambda,\mu}=(\lambda p_0,\lambda p_1, \lambda p_2,\mu)$. Substituting the coordinates of $P_{\lambda,\mu}$ into the equation of $\S$ we obtain
		\[
		0=\alpha_0\lambda r_0^2(P_{\lambda,\mu})+\beta_0\lambda^2 r_0(P_{\lambda,\mu})=\lambda(\alpha_0r_0^2(P_{\lambda,\mu})+\beta_0\lambda r_0(P_{\lambda,\mu})).
		\]
		Now since $r_0(P_1)\ne0$ and we want $(0,\mu)$ as double solution then $\alpha_0=0$. Iterating this process the equation of the surface becomes
		\[
		r_0r_1r_2+T(\beta_0r_0r_1+\beta_0^qr_1r_2+\beta_0^{q^2}r_0r_2)+T^2(\gamma_0r_0+\gamma_0^qr_1+\gamma_0^{q^2}r_2)=0.
		\] 
	\end{proof}

	We want to reduce the problem of counting the points in the form $(\alpha,\alpha^q,\alpha^{q^2})$ on the cubic surface to the problem of counting the points in the same form on a certain quadric. To achieve the result we apply the Cremona transform, call 
	\[
	z_1:=\frac{1}{r_1}\quad z_2:=\frac{1}{r_2}\quad z_3:=\frac{1}{r_3},
	\]
	dividing the equation of the surface by $r_1r_2r_3$ we obtain 
	\[
	\mathcal{Q}:\beta z_3+\beta^qz_1+\beta^{q^2}z_2+\gamma z_2z_3+\gamma^qz_1z_3+\gamma^{q^2}z_1z_2-1=0.
	\]
	Note that if $\gamma=0$ then $\mathcal{Q}$ collapse to a plane.
	
	\begin{proposition}
		The quadric surface $\mathcal{Q}$ is absolutely irreducible.
	\end{proposition}
	\begin{proof}
		If $\gamma=0$ there is nothing to prove, since $\mathcal{Q}$ is a plane. 
		Suppose $\gamma\ne0$ and that $\mathcal{Q}$ splits in the product of two planes $\pi_1$ and $\pi_2$, then
		\[
		\beta z_3+\beta^qz_1+\beta^{q^2}z_2+\gamma z_2z_3+\gamma^qz_1z_3+\gamma^{q^2}z_1z_2-1=(a_1z_1+a_2z_2+a_3z_3+a_4)(d_1z_1+d_2z_2+d_3z_3+d_4).
		\]
		From the identity principles of polynomials we get that $a_1d_1=a_2d_2=a_3d_3=0$
		Without loss of generality we can consider $a_1=a_2=d_3=0$ and then the equation becomes
		\[
		\beta z_3+\beta^qz_1+\beta^{q^2}z_2+\gamma z_2z_3+\gamma^qz_1z_3+\gamma^{q^2}z_1z_2-1=(a_3z_3+a_4)(d_1z_1+d_2z_2+d_4)
		\]
		and this cannot happen since in the right hand side of this equality we do not have the term $z_1z_2$.
	\end{proof}
	
	We want to count the points on the quadric $\mathcal{Q}$ in the form $(\delta,\delta^q,\delta^{q^2})$, where $\delta\in\F_{q^3}$. Writing down $\delta$ on the normal basis $\mathcal{B}$ we get $\delta=w_1\alpha+w_2\alpha^q+w_3\alpha^{q^2}$. Taking $w_1,w_2$ and $w_3$ as a set of variables (on $\F_q$) we obtain a $\F_q$-rational quadric surface and its $\F_q$-rational points are in one-to-one correspondence with the searched ones. 
	
	\[
	\begin{split}
	&\beta(w_1\alpha^{q^2}+w_2\alpha+w_3\alpha^q)+\beta^q(w_1\alpha+w_2\alpha^q+w_3\alpha^{q^2})+\beta^{q^2}(w_1\alpha^q+w_2\alpha^{q^2}+w_3\alpha)+\\
	&\gamma(w_1\alpha+w_2\alpha^q+w_3\alpha^{q^2})(w_1\alpha^q+w_2\alpha^{q^2}+w_3\alpha)+\gamma^q(w_1\alpha^{q^2}+w_2\alpha+w_3\alpha^q)(w_1\alpha^q+w_2\alpha^{q^2}+w_3\alpha)+\\
	&\gamma^{q^2}(w_1\alpha^{q^2}+w_2\alpha+w_3\alpha^q)(w_1\alpha+w_2\alpha^q+w_3\alpha^{q^2})-1=0.
	\end{split}
	\]
	The points we were looking for of the first surface are in one-to-one correspondence with the $\F_{q}$-rational points on the quadric surface above. It is widely known (see \cite[Section 15.3]{Hirschfeld1985}) that, in this case
	\begin{equation}
	\label{stima3punti}
	|S_1(\F_q)|= q^2+\eta q+1,\quad\eta\in\{0,1,2\}
	\end{equation}
	since the quadric surface $\mathcal{Q}$ is irreducible.

	\subsection{Four singular points}
	Call $P_1$, $P_2$, $P_3$ and $P_4$ the singular points of $\mathcal{S}_1$, applying Remark \ref{remark:singolarirazionali} we have the following possibilities:
	\begin{enumerate}[(i)]
		\item At least one among $P_1$, $P_2$, $P_3$ and $P_4$ is $\F_q$-rational;
		\item There are two couples of $\F_{q^2}$-rational and conjugates singular points.
		\item $P_1$, $P_2$, $P_3$ and $P_4$ are $\F_{q^4}$-rational and conjugates.
	\end{enumerate}
	
	If $(i)$ or $(ii)$ hold then we have already found out a good bound before respectively $|S_1(\F_{q})|\le q^2+3q-3$ and $|S_1(\F_{q})|\le q^2$, the last thing we have to do is show that $(iii)$ never holds. 
	
	\begin{proposition}
		\label{caso:singolari4}
		Case $(iii)$ never holds.
	\end{proposition}
	\begin{proof}
		In order to solve this problem we use a multivariate approach, calculating the elimination ideal with respect to all the variables less one. Consider the equations in (\ref{derivate:sup}): it is clear that, given $X_1$ and $X_2$, the value of $X_0$ is uniquely determined. For this reason we proceed with eliminating the variables $X_0$ and $X_1$ and we obtain the elimination ideal $I_{x_0,x_1}=\langle p_1,p_2\rangle$, where
		
		\[
		\begin{split}
		p_1(X_1)=&2X_1^5A^q + X_1^4B^q - 16X_1^3A^{q^2+q+1} - 8X_1^2A^{q^2+1}B^q - X_1^2B^{q^2+1} + 32X_1A{2q^2+q+2}\\&- 2X_1AB^{2q^2} - 2X_1A^{q^2}B^2 + 16A^{2q^2+2}B^q - 4A^{q^2+1}B^{q^2+1}\\
		p_2(X_1)=&(X_1^2 - 4A^{q^2+1})(X_1^4A^q + X_1^3B^q - 4X_1^2A^{q^2+q+1} + X_1^2D - 4X_1A^{q^2+1}B^q \\&+ x1B^{q^2+1} - 4A^{q^2+1}D + AB^{2q^2} + A^{q^2}B^2).
		\end{split}
		\]
		
		On the other hand, if we prooced eliminating the variables $X_0$ and $X_1$ we get the elimination ideal $J_{x_0,x_2}=\langle q_1,q_2\rangle$, where $q_1=p_1(X_2)^q$ and $q_2=p_2(X_2)^q$. The fact that the two ideals are generated by conjugate polynomials will continue to be true after symmetric annihilation of some of their terms. After further computations using the software MAGMA, which can be completely seen in \cite{tesi}, we get that one of the generators of $I_{x_1,x_2}$ is a polynomial of degree lower or equal to two, namely $f(X_1)$, and one of the generators of $J_{x_0,x_2}$ is $f(X_2)^q$. From this fact we get that the singularities of $\S_2$ are at most four and if this value is achieved then they belong (at most) to the field $\F_{q^6}$, which means that the singularities of $\S_1$ are at most in the field $\F_{q^2}$.
	\end{proof}

	\section{Case $r=3$ and $h=3$}
	Consider the case of the intersection over $\F_{q^3}$ between $\mathcal{X}$ and the curves $y=Ax^3+Bx^2+Cx+D$, $A,B,C,D\in\F_{q^3}$ and $A\ne 0$. 
	After doing similar computations to those done for the case $r=3$ and $h=2$ we arrive at an equation of a cubic surface $\widehat{\S}_1=\widehat{\S}_1(\overline{\F}_q)$ defined over $\F_q$, affinely equivalent to a surface $\widehat{\S}_2=\widehat{\S}_2(\overline{\F}_q)$ defined over $\F_{q^3}$, having equation 
	\[
	X_0X_1X_2=AX_0^3+A^qX_1^3+A^{q^2}X_2^3+BX_0^2+B^qX_1^2+B^{q^2}X_2^2+CX_0+C^qX_1+C^{q^2}X_2+E\\
	\]
	where $E=\text{T}(D)$.
	In this more general case $\widehat{\S}_1$ may be reducible, which can possibly increase the number of $\F_q$-rational points of $\widehat{\S}_1$, but on the other hand the reasonings done for $r=3$ and $h=2$ can be completely extended if $\widehat{\S}_1$ is irreducible, so we claim the following result.
	
	\begin{theorem}
		Let $r=h=3$ and consider the $\F_q$-rational cubic surface $\widehat{\S}_1$ associated to the intersections between $\mathcal{X}$ and $y=Ax^3+Bx^2+Cx+D$. If $\widehat{\S}_1$ is irreducible then
		\[
		|\widehat{\S}_1|\le q^2+7q+1.
		\]
	\end{theorem}
	
	\section{AG codes from the Norm-Trace curves}
	
	Consider the norm-trace curve over the field $\F_{q^3}$: since $r=3$, $\mathcal{X}$ has $N=q^{2r-1}=q^5$ $\F_{q^3}$-rational points in $\mathbb{A}^2(\F_{q^3})$. We also know that $\mathcal{L}_{\F_q}(2q^2P_\infty)=\{ay+bx^2+cx+d\,|\,a,b,c,d\in\F_{q^3}\}$. Considering the evaluation map 
	\[
	\begin{split}
	\mathrm{ev}: \,\, \mathcal{L}_{\F_{q^3}}(2q^2P_\infty)&\longrightarrow(\F_{q^3})^{q^5}\\
	f=\tilde{a}y+\tilde{b}x^2+\tilde{c}x+\tilde{d}&\longmapsto (f(P_1),\dots,f(P_N))\\
	\end{split}
	\]
	the associated one-point code will be $C_{\mathcal{L}}(D,2q^2P_{\infty})=\mathrm{ev}(\mathcal{L}_{\F_{q^3}}(2q^2P_\infty))$, where the divisor $D$ is the formal sum of all the $q^5$-rational affine points of $\mathcal{X}(\F_{q^3})$. The weight of a codeword associated to the evaluation of a function $f\in \mathcal{L}_{\F_{q^3}}(2q^2P_\infty)$ corresponds to 
	
	\[\mathrm{w}(\mathrm{ev}(f))=|\mathcal{X}(\F_{q^3})|-|\{\mathcal{X}(\F_{q^3})\cap\{\tilde{a}y+\tilde{b}x^2+\tilde{c}x+\tilde{d}=0\}\}|.\]
	
	\begin{enumerate}
		\item If $\tilde{a}=0$ then we have to study the common zeroes of $\tilde{b}x^2+\tilde{c}x+\tilde{d}$ and $\mathcal{X}(\F_{q^3})$.
		\begin{enumerate}
			\item if $\tilde{b}=\tilde{c}=\tilde{d}=0$ then $\mathrm{w}(\mathrm{ev}(f))= 0$;
			\item if $\tilde{b}=\tilde{c}=0$ and $\tilde{d}\ne0$ then $\mathrm{w}(\mathrm{ev}(f))=q^5$;
			\item if $\tilde{b}=0$ and $\tilde{c}\ne0$ then $\mathrm{w}(\mathrm{ev}(f))= q^5-q^2$;
			\item if $\tilde{c}\ne0$ and $\tilde{c}^2-4\tilde{b}\tilde{d}=0$ then $\mathrm{w}(\mathrm{ev}(f))= q^5-q^2$;
			\item otherwise $\mathrm{w}(\mathrm{ev}(f))= q^5-2q^2$.
		\end{enumerate}
		\item On the other hand, if $\tilde{a}\ne0$ then we have to study the common zeroes between $\mathcal{X}(\F_{q^3})$ and $\tilde{a}y+\tilde{b}x^2+\tilde{c}x+\tilde{d}$.
		\begin{enumerate}
			\item if $\tilde{b}=\tilde{c}=\tilde{d}=0$ then $\mathrm{w}(\mathrm{ev}(f))= q^5-1$;
			\item if $\tilde{b}=\tilde{c}=0$ and $\tilde{d}\ne0$ then $\mathrm{w}(\mathrm{ev}(f))=q^5-q^2$;
			\item if $\tilde{b}=0$ and $\tilde{c}\ne0$ then, applying Bézout theorem, we have that 
			$\mathrm{w}(\mathrm{ev}(f))~\ge~q^5~-~(q^2~+~q~+~1)$;
			\item otherwise, from what we said previously, $\mathrm{w}(\mathrm{ev}(f))\ge q^5-(q^2+7q+1)$.
		\end{enumerate}
		
	\end{enumerate}
	
	We can summarize our reasonings in the following result.
	
	\begin{theorem}
		Consider the norm-trace curve $\mathcal{X}$ over the field $\F_{q^3}$, $q\ge8$, and the AG code $C=C(D,2q^2P_{\infty})$ arising from $\mathcal{X}$, where $D=\sum_{P\in\mathcal{X}\left(\F_{q^3}\right)\setminus{P_{\infty}}}P$. Let $\{A_{\mathrm{w}}\}_{0\le \mathrm{w} \le q^5}$ be the weight distribution of $C$, then the following results hold
		\begin{enumerate}[(i)]
			\item $A_0=1$;
			\item The minimum distance of $C$ is $q^5-2q^2$;
			\item If $\mathrm{w}> q^5-2q^2$ and $A_\mathrm{w}\ne0$ then $\mathrm{w}\ge q^5-q^2-7q-1$;
		\end{enumerate}
	\end{theorem}
	
	In the cases $q<8$, i.e. $q=2,3,5,7$, the complete $\{A_w\}_{w\le q^5}$ can be determined with a computer and we do not report it here.
	
	\section*{Acknowledgements}
	The authors would like to thank the anonymous referee for their interesting and useful comments which permitted to improve the paper.
	This research was partially supported by the Italian National Group for Algebraic and Geometric Structures and their Applications (GNSAGA - INdAM). The results showed in this paper are included in M. Bonini’s Ph.D. thesis (supervised by the second author and G. Rinaldo).

	\begin{flushleft}
		Matteo Bonini\\
		Department of Mathematics,\\
		University of Trento,\\
		e-mail: {\sf matteo.bonini@unitn.it}
	\end{flushleft}
	
	\begin{flushleft}
		Massimiliano Sala\\
		Department of Mathematics,\\
		University of Trento,\\
		e-mail: {\sf massimiliano.sala@unitn.it}
	\end{flushleft}
	

\begin{thebibliography}{99}
		
		\bibitem{BR2013} E. Ballico, A. Ravagnani. On the duals of geometric Goppa codes from norm-trace curves. Finite Fields Appl. {\bf 20}, 30-39, (2013).
		
		\bibitem{BB2018}{D. Bartoli, M. Bonini, \textit{Minimum weight codewords in dual algebraic-geometric codes from the Giulietti-Korchm\'aros curve}, Des. Codes Cryptography, to appear (https://doi.org/10.1007/s10623-018-0541-y) (2018).}
		
		\bibitem{BMZ2} D. Bartoli, M. Montanucci, G. Zini. \textit{AG codes and AG quantum codes from the GGS curve}, Des. Codes and Cryptography {\bf 86}(10), 2315-2344 (2018).
		
		\bibitem{BMZ1} D. Bartoli, M. Montanucci, G. Zini. \textit{Multi point AG codes on the GK maximal curve}, Des. Codes and Cryptography {\bf 86}(1), 161-177 (2018).
		
		\bibitem{tesi} M. Bonini. Intersections of Algebraic Curves and their link to the weight enumerators of Algebraic-Geometric Codes. Ph.D. Thesis (2019).
		
		\bibitem{B2013} T.D. Browning. \textit{The Lang-Weil estimate for cubic hypersurfaces}, Canad. Math. Bull. 56, 500--502 (2013).
		
		\bibitem{CT1988} D.F. Coray, M.A. Tsfasman. Arithmetic on singular Del Pezzo surfaces. Proceedings of the London Mathematical Society {\bf 3} (1), 25-87(1988).
		
		\bibitem{C2012}{A. Couvreur, \textit{The dual minimum distance of arbitrary-dimensional algebraic–geometric codes}. Journal of Algebra \textbf{350}, 84-107 (2012).}
		
		\bibitem{DDK209} G. Donati, N.Durante, G.Korchmáros. On the intersection pattern of a unital and an oval in $PG(2,q^2)$. Finite Fields and Their Applications {\bf 15}, 785--795 (2009).
		
		
		\bibitem{G2003} O. Geil,(2003). On codes from norm–trace curves. Finite fields and their Applications {\bf 9} (3), 351--371.
		
		\bibitem{Goppa81} V.D. Goppa. Codes on algebraic curves. Dokl. Akad. NAUK SSSR {\bf 259}, 1289--1290 (1981).
		
		\bibitem{Goppa82} V.D. Goppa. Algebraic-geometric codes. Izv. Akad. NAUK SSSR {\bf 46}, 75--91 (1982).
		
		\bibitem{Hirschfeld1985} J.W.P. Hirschfeld, Finite projective spaces of three dimensions. Oxford University Press, 1985.
		
		\bibitem{Hirschfeld1998} J.W.P. Hirschfeld Projective Geometries Over Finite Fields. Oxford Mathematical Monographs. New York: Oxford University Press, 1998.
		
		\bibitem{KL2018}B. Kim, Y. Lee. The minimum weights of two-point AG codes on norm-trace curves. Finite Fields and their Applications {\bf 53}, 113-–139.
		
		\bibitem{LW1954} S. Lang, A. Weil. Number of points of varieties in finite fields, Amer. J. Math.
		{\bf 76}, 819--827 (1954).
		
		\bibitem{LN1997} R. Lidl, H. Niederreiter. Finite fields. Vol. 20. Cambridge university press, 1997.
		
		\bibitem{M1986} Y.I. Manin, Cubic forms: algebra, geometry, arithmetic. Vol. 4. Elsevier (1986).
		
		\bibitem{MPS2014} C, Marcolla, M. Pellegrini, M. Sala,  On the Hermitian curve and its intersection with some conics, Finite Fields and Their Applications 28 (2014) 166--187.
		
		\bibitem{MPS2016} C. Marcolla, M. Pellegrini, M. Sala.  On the small-weight codewords of some Hermitian codes. J. Symbolic Comput. {\bf 73},  27--45 (2016).
		
		\bibitem{MR2018} C. Marcolla, M. Roggero. Minimum-weight codewords of the Hermitian codes are supported on complete intersections, Journal of Pure and Applied Algebra, to appear (https://doi.org/10.1016/j.jpaa.2018.12.007).
		
		\bibitem{M2004} G.L. Matthews, \textit{Codes from the Suzuki function field}, IEEE Transactions on Information Theory {\bf 50} (12), 3298-3302 (2004).
		
		\bibitem{FMT2013}J.I. Farrán, C. Munuera, G. C. Tizziotti, F. Torres. Gröbner basis for norm-trace codes. Journal of Symbolic Computation {\bf 48}, 54--63 (2013).
		
		\bibitem{MTT2008}C. Munuera, G. C. Tizziotti, F. Torres. Two-point codes on Norm-Trace curves. Coding Theory and Applications. Springer, Berlin, Heidelberg. 128--136 (2008).
		
		\bibitem{Sala} M. Sala.  Gr\"obner basis techniques to compute weight distributions of shortened cyclic codes. J. Algebra Appl. {\bf 6}(3),  403--404 (2007).
		
		\bibitem{Stichtenoth1988} H. Stichtenoth.  A note on Hermitian codes over $GF(q^2)$. IEEE Trans. Inf. Theory {\bf 34}(5), 1345--1348 (1988).
		
		\bibitem{S2009} H. Stichtenoth. Algebraic function fields and codes. Graduate Texts in Mathematics
		{\bf 254}, Springer, Berlin (2009).
		
		\bibitem{Tiersma1987} H.J. Tiersma. Remarks on codes from Hermitian curves. IEEE Trans. Inf. Theory {\bf 33}(4), 605--609 (1987).
		
		\bibitem{XL2000} C.P. Xing, S. Ling. A class of linear codes with good parameters from algebraic curves. IEEE Trans. Inf. Theory {\bf 46}(4), 1527--1532 (2000).
		
		\bibitem{YK1992} K. Yang, P.V. Kumar. On the true minimum distance of Hermitian codes. Coding Theory and Algebraic Geometry {\bf 1518}, Lecture Notes in Math., 99--107 (1992).
		
		
		
		
	\end{thebibliography}
\end{document}